\documentclass[11pt,a4paper]{article}
\usepackage{latexsym}
\usepackage{amsfonts}
\usepackage{amssymb}
\usepackage{amsbsy}
\usepackage{amsmath}
\usepackage{amsthm}

\usepackage{ifthen}

\usepackage{url}
\usepackage{paralist}
\usepackage{graphicx}
\usepackage{array}
\usepackage{arydshln}
\newcommand\atp[2]{\genfrac{(}{)}{0pt}{}{#1}{#2}}
\newcommand\diag{\operatorname{diag}}
\newcommand\sech{\operatorname{sech}}

\newtheoremstyle{conc}%
  {5pt}{5pt}{\itshape}{0pt}{\bfseries}{.}{ }{}%
\theoremstyle{conc}
\newtheorem{thm}{Theorem}[section]
\newtheorem{lem}[thm]{Lemma}

\newtheorem{con}[thm]{Conjecture}

\newtheoremstyle{nap}%
  {5pt}{5pt}{\normalfont}{0pt}{\bfseries}{.}{ }{}%

\begin{document}

\date{\today}

\title{Estimates for the Spectral Condition Number of Cardinal B-Spline
  Collocation Matrices (Long version)\thanks{This work was supported by grant
  037--1193086--2771 by the Ministry of Science, Education and Sports
  of the Republic of Croatia.}}

\author{%
Vedran Novakovi\'{c}%
  \thanks{Faculty of Mechanical Engineering and Naval Architecture,
  University of Zagreb, Ivana Lu\v{c}i\'{c}a 5, 10000, Croatia,
  e-mail: {\tt venovako@fsb.hr}},\hbox{\ }
Sanja Singer%
  \thanks{Faculty of Mechanical Engineering and Naval Architecture,
  University of Zagreb, Ivana Lu\v{c}i\'{c}a 5, 10000, Croatia,
  e-mail: {\tt ssinger@fsb.hr}}\hbox{\ }
and Sa\v{s}a Singer%
  \thanks{Department of Mathematics, University of Zagreb,
  Bijeni\v{c}ka cesta 30, 10000 Zagreb, Croatia,
  e-mail: {\tt singer@math.hr}}}

\maketitle

\markboth{Novakovi\'{c}, Singer, and Singer}
  {Estimates for the Condition of Cardinal B-Spline Collocation Matrices}

\begin{abstract}
  The famous de Boor conjecture states that the condition of the polynomial
  B-spline collocation matrix at the knot averages is bounded independently of
  the knot sequence, i.e., depends only on the spline degree.

  For highly nonuniform knot meshes, like geometric meshes, the conjecture is
  known to be false. As an effort towards finding an answer for uniform
  meshes, we investigate the spectral condition number of cardinal B-spline
  collocation matrices. Numerical testing strongly suggests that the
  conjecture is true for cardinal B-splines.
\end{abstract}


\vspace*{6pt}
\noindent
{\bf\small Keywords}: {\rm \small cardinal splines, collocation matrices,
condition, T\"{o}plitz matrices, circulants}


\vspace*{6pt}
\noindent
{\bf\small AMS subject classifications}: {\rm\small 65D07, 65D05, 65F35, 15A12}

%
%
\section{Introduction}
%
%
We consider the classical Lagrange function interpolation problem in
the following ``discrete'' setting. Let $\tau_1, \ldots, \tau_\nu$
be a given set of mutually distinct interpolation nodes, and let
$f_1, \ldots, f_\nu$ be a given set of ``basis'' functions.

For any given function $g$ we seek a linear combination of the basis functions
that interpolates $g$ at all interpolation nodes,
\begin{displaymath}
  \sum_{j = 1}^{\nu} y_j f_j(\tau_i) = g(\tau_i), \quad i = 1, \ldots, \nu.
\end{displaymath}
The coefficients $y_j$ can be computed by solving the linear system
\begin{displaymath}
  A y = g,
\end{displaymath}
where $A$ is the so-called \textit{collocation matrix\/} containing the
values of the basis functions at the interpolation nodes
\begin{displaymath}
  a_{ij} = f_j(\tau_i).
\end{displaymath}
If the basis functions are linearly independent on
$\{ \tau_1, \ldots, \tau_{\nu} \}$, the matrix $A$ is nonsingular, and
the interpolation problem has a unique solution for all functions $g$.

The sensitivity of the solution is then determined by the condition number
$\kappa_p(A)$ of the collocation matrix
\begin{equation}
  \kappa_p(A) = \| A \|_p \| A^{-1} \|_p,
  \label{1.1}
\end{equation}
where $\| \ \|_p$ denotes the standard operator $p$-norm, with
$1 \leq p \leq \infty$.

The polynomial B-splines of a fixed degree $d$ are frequently used
as the basis functions in practice. Such a basis is uniquely determined by
a given multiset of knots that defines the local smoothness of the basis
functions. The corresponding collocation matrix is always totally
nonnegative (see~\cite{deBoor-2001,Schumaker-2007}), regardless of the
choice of the interpolation nodes.

A particular choice of nodes is of special interest, both in theory and
in practice, for shape preserving approximation. When the nodes are located
at the so-called Greville sites, i.e., at the \textit{knot averages\/},
the interpolant has the variation diminishning property. Moreover, for
low order B-spline interpolation, it can be shown that $\kappa_\infty(A)$
is bounded \textit{independently\/} of the knot sequence~\cite{deBoor-2001}.

On these grounds, in 1975, Carl de Boor~\cite{deBoor-75} conjectured that
the interpolation by B-splines of degree $d$ at knot averages is bounded
by a function that depends only on $d$, regardless of the knots themselves.
In our terms, the conjecture says that $\kappa_\infty(A)$ or, equivalently,
$\| A^{-1} \|_\infty$ is bounded by a function of $d$ only.
The conjecture was disproved by Rong--Qing Jia~\cite{Jia-88} in 1988.
He proved that, for geometric meshes, the condition number $\kappa_\infty(A)$
is not bounded independently of the knot sequence, for degrees $d \geq 19$.

Therefore, it is a natural question whether exists any class of meshes for
which de Boor's conjecture is valid. Since geometric meshes are highly
nonuniform, the most likely candidates for the validity are uniform meshes.

Here we discuss the problem of interpolation at knot averages by B-splines
with \textit{equidistant simple\/} knots. The corresponding B-splines are
symmetric on the support, and have the highest possible smoothness.
It is easy to see that the condition of this interpolation does not depend
on the knot spacing $h$, and we can take $h = 1$. So, just for simplicity,
we shall consider only the cardinal B-splines, i.e., B-splines with simple
knots placed at successive integers.
It should be stressed that the only free parameters in this problem are
the \textit{degree\/} $d$ of the B-splines and the \textit{size\/} $\nu$
of the interpolation problem. Our aim is to prove that the condition of $A$
can be bounded independently of its order.

The corresponding collocation matrices $A$ are symmetric, positive definite,
and most importantly, T\"{o}plitz. But, it is not easy to compute the
elements of $A^{-1}$, or even reasonably sharp estimates of their magnitudes.
So, the natural choice of norm in (\ref{1.1}) is the spectral norm
\begin{equation}
  \kappa_2(A) = \frac{\sigma_{\max}(A)}{\sigma_{\min}(A)},
  \label{1.2}
\end{equation}
where $\sigma_{\max}(A)$, and $\sigma_{\min}(A)$ denote the largest and the
smallest singular value of $A$, respectively.

For low spline degrees $d \leq 6$, the collocation matrices are also strictly
diagonally dominant, and it is easy to bound $\kappa_2(A)$ by a constant.
Consequently, de Boor's conjecture is valid for $d \leq 6$.

For higher degrees, the condition number can be estimated by embedding
the T\"{o}plitz matrix $A$ into circulant matrices of higher orders.
The main advantage of this technique, developed by Davis~\cite{DavisP-94}
and Arbenz~\cite{Arbenz-91}, lies in the fact that the eigenvalues of
a circulant matrix are easily computable. The final bounds for
$\kappa_2(A)$ are obtained by using the Cauchy interlace theorem for
singular values (see~\cite{Horn-Johnson-91} for details), to bound
both singular values in (\ref{1.2}).

The paper is organized as follows. In Section 2 we briefly review some
basic properties of cardinal B-splines. The proof of de Boor's conjecture
for low degree ($d \leq 6$) cardinal B-splines is given in Section 3.
In Section 4 we describe the embedding technique and derive the estimates
for $\kappa_2(A)$.

	Despite all efforts, we are unable to prove de Boor's conjecture
in this, quite probably, the easiest case.
The final section contains the results of numerical testing that strongly
support the validity of the conjecture, as well as some additional
conjectures based on these test results.

%
%
\section{Properties of cardinal splines}
%
%
Let $x_i = x_0 + ih$, for $i = 0, \ldots, n$, be a sequence of simple
uniformly spaced knots. This sequence determines a unique sequence of
normalized B-splines $N_0^d, \ldots, N_{n - d - 1}^d$ of degree $d$, such
that the spline $N_i^d$ is non-trivial only on the interval
$\langle x_i, x_{i + d + 1} \rangle$.

Each of these B-splines can obtained, by translation and scaling,
from the basic B-spline $Q^d$ with knots $i = 0, \ldots, d + 1$,
\begin{equation}
  Q^{d}(x) = \frac{1}{(d + 1)!} \sum_{i = 0}^{d + 1}
    (-1)^i \atp{d + 1}{i} (x - i)_{+}^d.
  \label{2.1}
\end{equation}
Here, $(x - i)_{+}^d$ denotes the truncated powers
$(x - i)_{+}^d = (x - i)_{}^d (x - i)_{+}^0$, for $d > 0$, while
\begin{displaymath}
  (x - i)_{+}^0 = \begin{cases}
    0, & \hbox{$x < i$,} \\
    1, & \hbox{$x \geq i$.}
  \end{cases}
\end{displaymath}
The normalized version of the basic spline is defined as
\begin{equation}
  N^d(x) = (d + 1) Q^d(x).
\label{2.2}
\end{equation}
From (\ref{2.1}) and (\ref{2.2}), we obtain the \textit{normalized\/}
B-spline basis $\{N_i^d\}$:
\begin{displaymath}
  N_i^d(x) = N_{}^d \left( \frac{x - x_i}{h} \right).
\end{displaymath}

If the interpolation nodes $\tau_i$ are located at the knot averages, i.e.,
\begin{equation}
  x_i^{\ast} = \frac{x_{i + 1} + \cdots + x_{i + d}}{d}
    = x_i + h \frac{d + 1}{2}, \quad i = 0, \ldots, n - d - 1,
  \label{2.3}
\end{equation}
then
\begin{displaymath}
  x_i^{} < x_i^{\ast} < x_{i + d + 1}^{},
\end{displaymath}
and the Sch\"{o}nberg--Whitney theorem~\cite{deBoor-2001} guarantees that
the collocation matrix is nonsingular. Moreover, this matrix is totally
nonnegative~\cite{Karlin-68}, i.e., all of its minors are nonnegative.
Due to symmetry of B-splines on uniform meshes, the collocation matrices
are also symmetric and T\"{o}plitz. So, we can conclude that cardinal
B-spline collocation matrix is T\"{o}plitz, symmetric and positive definite.

It is easy to show that the elements of the collocation matrix do not
depend on the step-size $h$ of the uniform mesh, so we take the simplest
one with $h = 1$ and $x_i = i$. Such B-splines are called cardinal.
The interpolation nodes (\ref{2.3}) are integers for odd degrees, while for
even degrees, the interpolation nodes are in the middle of the two
neighbouring knots of the cardinal B-spline.
\begin{figure}[hbt]
  \begin{center}
    \includegraphics[width=5.75cm]{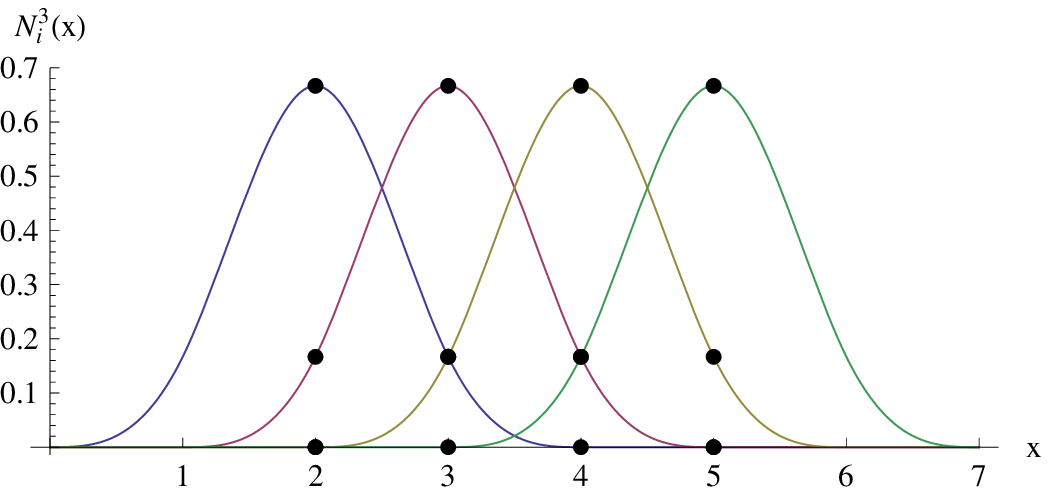} \qquad
    \includegraphics[width=5.75cm]{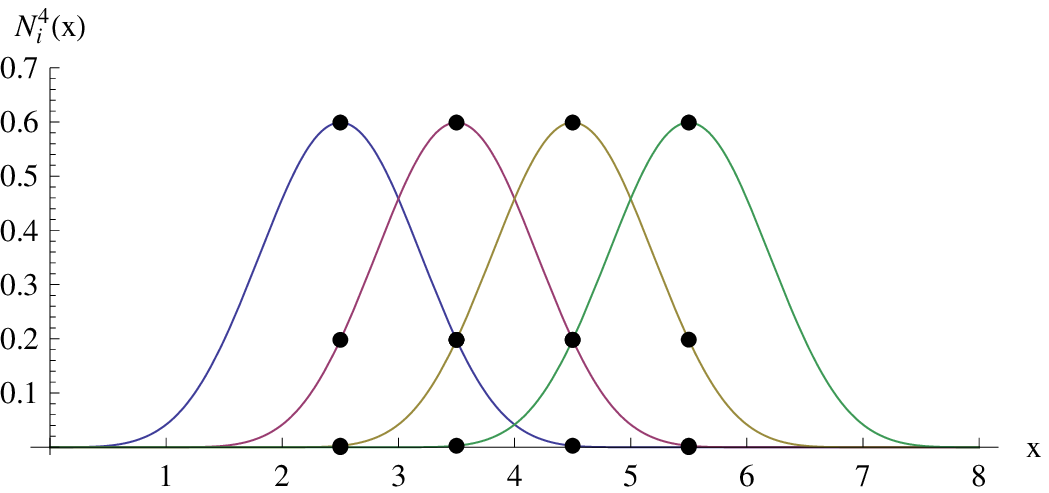}
  \end{center}
  \caption{The first four cardinal splines $N_i^d$ of degree $d = 3$ and
    $4$, respectively. Big black dots denote the spline values at
    the knot averages.}
  \label{fig2.1}
\end{figure}

The normalized basic cardinal spline $N^d$ suffices to determine all
basis function values at the interpolation nodes
\begin{displaymath}
  N_i^d (x_j^{\ast}) = N^d(x_{j-i}^{\ast}).
\end{displaymath}

The general de Boor--Cox reccurence relation~\cite{deBoor-2001}, written
in terms of the degree of a spline is:
\begin{equation}
  N^d(x) = \frac{x N^{d - 1}(x) + (d + 1 - x) N^{d - 1}(x - 1)}{d}.
  \label{2.4}
\end{equation}
Note that the elements of a collocation matrix are rational, because
the interpolation nodes are rational, and the de Boor--Cox recurrence
formula (\ref{2.4}) involves only basic arithmetic operations on
rational coefficients. These elements are therefore exactly computable
in (arbitrary precise) rational arithmetic.

%
%
\section{Low degree cardinal B-splines}
%
%
Let $t_i^d = N^d(x_i^{\ast})$ denote the values of cardinal B-splines
at knot averages (see (\ref{2.3})). Then, the cardinal B-spline collocation
matrix $A$ with interpolation nodes $x_i^{\ast}$ is a banded T\"{o}plitz
matrix of order $n - d$, to be denoted by
\begin{equation}
  T_n^d = \begin{bmatrix}
    t_0^d    & \cdots & t_r^d  &        & 0      \\
    \vdots   & \ddots &        & \ddots &        \\
    t_r^d    &        & \ddots &        & t_r^d  \\
             & \ddots &        & \ddots & \vdots \\
    0        &        & t_r^d  & \cdots & t_0^d,
  \end{bmatrix}, \quad
  r = \left\lfloor \frac{d}{2} \right\rfloor.
  \label{3.1}
\end{equation}
The matrix $T_n^d$ is represented by its first row, usually called the symbol,
\begin{displaymath}
  t = (t_0^d, \ldots, t_r^d, 0, \ldots, 0), \quad t \in \mathbb{R}^{n-d}.
\end{displaymath}

It is useful to note that each B-spline of degree $d > 0$ is a unimodal
function, i.e., it has only one local maximum on the support. In the
case of cardinal B-splines, we have already concluded that the splines
are symmetric, and therefore the maximum values of $N^d$ is attained
at the middle of the support, for $x = (d + 1) /2$. The maximum value is
\begin{displaymath}
  N^d\left( \frac{d + 1}{2} \right) = N^d(x_0^{\ast}) = t_0^d.
\end{displaymath}
Furthermore, unimodality implies that the values of the spline $N^d$
are decreasing in the interval $[ (d + 1) /2, d + 1 ]$, so
\begin{equation}
  t_0^d > t_1^d > \cdots > t_r^d.
  \label{3.2}
\end{equation}

To esitmate the condition number of a cardinal B-spline we need to
bound both the minimal and the maximal singular value of $T_n^d$.  For a
symmetric and positive definite matrix, the singular values are
eigenvalues. Therefore, the bounds for the eigenvalues of $T_n^d$ are
sought for. From the Ger\v{s}gorin bound for the eigenvalues, and the
partition of unity of the B-spline basis, we obtain an upper bound
for $\lambda_{\max}(T_n^d)$
\begin{equation}
  \lambda_{\max}(T_n^d) \leq t_0^d + 2 (t_1^d + \cdots + t_r^d) = 1.
  \label{3.3}
\end{equation}
Similarly, we also obtain a lower bound for $\lambda_{\min}(T_n^d)$,
\begin{equation}
  \lambda_{\min}(T_n^d) \geq t_0^d - 2 (t_1^d + \cdots + t_r^d) = 2t_0^d - 1,
  \label{3.4}
\end{equation}
which is sensible only if $T_n^d$ is strictly diagonally dominant. Strict
diagonal dominance is achieved only for B-spline degrees
$d = 1, \ldots, 6$ (easily verifiable by a computer).
The corresponding Ger\v{s}gorin bounds are presented in Table
\ref{tbl3.1}.  This directly proves de Boor's conjecture for low order
B-splines.
\begin{table}[hbt]
\begin{center}
\newcolumntype{I}{!{\vrule width 0.6pt}}
\newlength\savedwidth
\newcommand\whline{\noalign{\global\savedwidth\arrayrulewidth
  \global\arrayrulewidth 0.6pt}%
  \hline
  \noalign{\global\arrayrulewidth\savedwidth}}

\renewcommand{\arraystretch}{1.25}
\begin{tabular}{cIccccc}
\whline
$n\backslash d$ & 2        & 3        & 4        & 5        & 6 \\
\whline
  64  & 1.998136 & 2.994873 & 4.785918 & 7.466648 & 11.727897 \\
 128  & 1.999541 & 2.998757 & 4.796641 & 7.492176 & 11.785901 \\
 256  & 1.999886 & 2.999694 & 4.799180 & 7.498105 & 11.799106 \\
 512  & 1.999971 & 2.999924 & 4.799797 & 7.499534 & 11.802256 \\
1024  & 1.999993 & 2.999981 & 4.799950 & 7.499884 & 11.803026 \\
2048  & 1.999998 & 2.999995 & 4.799987 & 7.499971 & 11.803216 \\
\whline
$\textrm{GB}(d)$ & 2 & 3 & $\frac{96}{19} \approx 5.052632$ & 10 & $\frac{5760}{127} \approx 45.354331$ \\
\whline
\end{tabular}
\end{center}
\caption{Comparison of the actual condition numbers $\kappa_2(T_n^d)$, for
  $d = 2, \ldots, 6$, $n = 64, \ldots, 2048$, and the bounds $\textrm{GB}(d)$
  for $\kappa_2(T_n^d)$, obtained by the Ger\v{s}gorin circle theorem.}
\label{tbl3.1}
\end{table}

Note that in the case of tridiagonal T\"{o}plitz matrices, i.e.~for
$d = 2, 3$, and, thus, $r = 1$ in (\ref{3.1}), the exact eigenvalues are
also known (see B\"{o}ttcher--Grudsky~\cite{Boettcher-Grudsky-2005})
\begin{displaymath}
  \lambda_k (T_n^d) = t_0^d + 2 t_1^d \cos \frac{\pi k}{n - d + 1}, \quad
  d = 2, 3, \quad k = 1, \ldots, n - d.
\end{displaymath}
The largest and the smallest eigenvalue can then be uniformly bounded by
\begin{align*}
  \lambda_{\max} (T_n^d) & = t_0^d + 2 t_1^d \cos \frac{\pi}{n - d + 1}
    < t_0^d + 2 t_1^d, \\
  \lambda_{\min} (T_n^d) & = t_0^d - 2 t_1^d \cos \frac{\pi}{n - d + 1}
    > t_0^d - 2 t_1^d > 0,
\end{align*}
These uniform bounds are somewhat better than those obtained by the
Ger\v{s}gorin circles.
%
%
\section{Embeddings of T\"{o}plitz matrices into circulants}
%
%
When the degree of a cardinal B-spline is at least $7$, the eigenvalue bounds
for T\"{o}plitz matrices can be computed by circulant embeddings. First,
we will introduce the smallest possible circulant embedding, and give its
properties. Then we will present some other known embeddings, with positive
semidefinite circulants.

To obtain a bound for $\lambda_{\min}(T_n^d)$, the collocation matrix $T_n^d$
is to be embedded into a circulant
\begin{equation}
  C_m^d = \left[ \begin{array}{ccccccc:ccc}
      \!        &        &       &               &       &        &        & t_r^d  & \cdots & t_1^d  \! \\
      \!        &        &       &               &       &        &        &        & \ddots & \vdots \! \\
      \!        &        &       &               &       &        &        &        &        & t_r^d  \! \\
      \!        &        &       & \smash{T_n^d} &       &        &        &        &        &        \! \\
      \!        &        &       &               &       &        &        & t_r^d  &        &        \! \\
      \!        &        &       &               &       &        &        & \vdots & \ddots &        \! \\
      \!        &        &       &               &       &        &        & t_1^d  &        & t_r^d  \! \\
      \hdashline
      \! t_r^d  &        &       &               & t_r^d & \cdots & t_1^d  & t_0^d  & \ddots & \vdots \! \\
      \! \vdots & \ddots &       &               &       & \ddots &        & \ddots & \ddots & t_1^d  \! \\
      \! t_1^d  & \cdots & t_r^d &               &       &        & t_r^d  & \cdots & t_1^d  & t_0^d  \!
    \end{array} \right], \quad
  m = n - d + r.
  \label{4.1}
\end{equation}
It is obviously a T\"{o}plitz matrix with the
following symbol
\begin{displaymath}
  t = (t_0^d, \ldots, t_r^d, 0, \ldots, 0, t_r^d, \ldots, t_1^d),
  \quad t \in \mathbb{R}^{n - d + r}.
\end{displaymath}
This circulant $C_m^d$ is called a periodization of $T_n^d$ by B\"{o}ttcher
and Grudsky~\cite{Boettcher-Grudsky-2005}.

The bounds (\ref{3.3})--(\ref{3.4}) for the eigenvalues of $T_n^d$
are also valid for $C_m^d$. Moreover, $C_m^d$ is doubly stochastic,
always having $\lambda_{\max}(C_m^d) = 1$ as its largest eigenvalue.
Interestingly enough, the upper bound (\ref{3.3}) is attained here
(the Ger\v{s}gorin bounds are rarely so sharp).

The symmetry of $T_n^d$ immediately implies the symmetry of $C_m^d$, and we
can conclude that the eigenvalues of $C_m^d$ are real, but not
necessarily positive. For symmetric matrices, the singular values are,
up to a sign, equal to the eigenvalues,~so
\begin{equation}
  \sigma_i(C_m^d) = |\lambda_i(C_m^d)|.
  \label{4.2}
\end{equation}
If the eigenvalues of the circulant $C_m^d$ are known, the spectrum of
embedded $T_n^d$ can be bounded by the Cauchy interlace theorem for singular
values, applied to $C_m^d$.
\begin{thm}[Cauchy interlace theorem]
  Let $C \in \mathbb{C}^{m \times n}$ be given, and let $C_\ell$
  denote a submatrix of $C$ obtained by deleting a total of $\ell$ rows
  and/or $\ell$ columns of $C$. Then
  \begin{displaymath}
    \sigma_k(C) \geq \sigma_k(C_\ell) \geq \sigma_{k+\ell}(C), \quad
    k = 1, \ldots, \min \{ m, n \},
  \end{displaymath}
  where we set $\sigma_j(C) \equiv 0$ if $j > \min \{m, n \}$.
\end{thm}

The proof can be found, for example, in~\cite[page 149]{Horn-Johnson-91}.

If we delete the last $r$ rows and columns of $C_m^d$, we obtain $T_n^d$.
The Cauchy interlace theorem will then give useful bounds for
$\sigma_{\min}(T_n^d) = \lambda_{\min}(T_n^d)$, provided that $C_m^d$
is nonsingular. Moreover, if we delete more than $r$ last rows and
columns of $C_m^d$, we obtain bounds for T\"{o}plitz matrices $T_k^d$,
of order $k - d$, for $k \leq n$,
\begin{equation}
  \kappa_2(T_k^d) = \frac{\sigma_{\max}(T_k^d)}{\sigma_{\min}(T_k^d)}
  \leq \frac{\sigma_{\max}(C_m^d)}{\sigma_{\min}(C_m^d)}
  = \frac{1}{\min_j|\lambda_j(C_m^d)|}.
  \label{4.3}
\end{equation}
Now we need to calculate the smallest singular value of $C_m^d$, and show
that it is non-zero.

The eigendecomposition of a circulant matrix is well-known
(see~\cite{DavisP-94,Arbenz-91}). A circulant $C$ of order $m$,
defined by the symbol $(c_0, \ldots, c_{m - 1})$, can be written as
\begin{displaymath}
  C = \sum_{j = 0}^{m - 1} c_j \Pi^j,
\end{displaymath}
where
\begin{displaymath}
  \Pi = \begin{bmatrix}
    \; 0 \;\; & 1              &                &   \\[3pt]
              & \smash{\ddots} & \smash{\ddots} &   \\
              &                & \smash{\ddots} & \;\; 1 \;\; \\
    \; 1 \;\; &                &                & \;\; 0 \;\; \\
  \end{bmatrix}.
\end{displaymath}

\vspace*{3pt}
\noindent
The spectral decomposition of $\Pi$ is $\Pi = F \Omega F^{\ast}$, where
\begin{displaymath}
  \Omega = \diag (1, \omega, \omega^2, \ldots, \omega^{m - 1}), \quad
  \omega = \frac{2\pi i}{m}, \quad i = \sqrt{-1},
\end{displaymath}
while
\begin{displaymath}
  F_{j, k} = \frac{1}{\sqrt{m}} \omega^{kj}, \quad
  0 \leq k, j \leq m - 1.
\end{displaymath}
Hence, $C$ can be decomposed as
\begin{displaymath}
  C = F \Lambda F^{\ast}, \quad
  \Lambda = \diag(\lambda_0, \ldots, \lambda_{m - 1})
  = \sum_{j = 0}^{m - 1} c_j \Omega^j.
\end{displaymath}
The eigenvalues of a real symmetric circulant $C$ are real, and given by
\begin{equation}
  \lambda_k(C) = c_0 + \sum_{j = 1}^{m - 1} c_j \cos \frac{2\pi k j}{m},
  \quad k = 0, \ldots, m - 1.
  \label{4.4}
\end{equation}
They can also be viewed as the discrete Fourier transform (DFT) of the symbol
$(c_0, \ldots, c_{m - 1})$.

For real and symmetric $C$, i.e., when $c_k = c_{m - k}$, for
$k = 1, \ldots, m - 1$, from (\ref{4.4}) it also follows that
\begin{displaymath}
  \lambda_k(C) = c_0 + \sum_{j = 1}^{m - 1} c_j \cos \frac{2\pi k j}{m}
  = c_0 + \sum_{j = 1}^{m - 1} c_{m - j} \cos \frac{2\pi k (m - j)}{m}
  = \lambda_{m - k}(C).
\end{displaymath}
So, all the eigenvalues, except $\lambda_0(C)$, and possibly
$\lambda_{\frac{m}{2}}(C)$, for even $m$, are multiple.

Therefore, the eigenvalues of the circulant $C_m^d$ from (\ref{4.1}) are
\begin{equation}
  \lambda_k(C_m^d) = t_0^d + 2 \sum_{j = 1}^{r} t_j^d \cos \frac{2\pi k j}{m},
  \quad k = 0, \ldots, m - 1.
  \label{4.6}
\end{equation}

For prime orders $m$, the nonsingularity of $C_m^d$ is a consequence of
the following theorem from~\cite{Geller-Kra-Popescu-Simanca-2004}.
\begin{thm}[Geller, Kra, Popescu and Simanca]
  Let $m$ be a prime number. Assume that the circulant $C$ of order
  $m$ has entries in $\mathbb{Q}$. Then  $\det C = 0$ if and only if
  \begin{displaymath}
    \lambda_0 = \sum_{j = 0}^{m - 1} c_i = 0,
  \end{displaymath}
  or all the symbol entries $c_i$ are equal.
\label{thmGKPS}
\end{thm}

If $m$ is prime, then we must have $\det C_m^d \neq 0$, since (\ref{3.2})
implies that $c_i$'s are not equal, and from (\ref{4.6}) we get
\begin{displaymath}
  \lambda_0 = c_0 + 2 \sum_{j = 1}^r c_j = 1 \neq 0.
\end{displaymath}
Theorem \ref{thmGKPS} suggests how to get the nonsingular embedding of
$T_n^d$. First, $T_n^d$ should be embedded into the T\"{o}plitz matrix
$T_p^d$, of order $p - d$, where $p \geq n$ is chosen so that
$m = p - d + r$ is a prime number. Then, $T_p^d$ is embedded into the
circulant $C_m^d$.

The other possibility is to embed $T_n^d$ into the smallest circulant matrix
$C_m^d$, as in (\ref{4.1}), and calculate its eigenvalues from (\ref{4.6}),
in hope that $C_m^d$ is nonsingular. In this case, extensive numerical
testing suggests that $C_m^d$ is always positive definite, but we have not
been able to prove it.

There are also several other possible embeddings that guarantee the
positive semidefiniteness of the circulant matrix $C$.

The first one, constructed by Dembo, Mallows and Shepp
in~\cite{Dembo-Mallows-Shepp-89}, ensures that the positive definite
T\"{o}plitz matrix $T$, of order $n$, can be embedded in the positive
semidefinite circulant $C$, of order $m$, where
\begin{equation}
  m \geq 2 \left(n + \kappa_2(T) \frac{n^2}{\sqrt{6}} \right).
  \label{4.7}
\end{equation}
A few years later, Newsam and Dietrich~\cite{Newsam-Dietrich-94} reduced the
size of the embedding to
\begin{equation}
  m \geq 2 \sqrt{6n^2 + \kappa_2(T) \frac{3 \cdot 2^{11/2}
    \, n^{5/2}}{5^{5/2}}}.
  \label{4.8}
\end{equation}
Note that among all positive semidefinite matrices $C$ of order
greater or equal $m$, we can choose one of prime order. This embedding
will be positive definite according to Theorem \ref{thmGKPS}. It is
obvious that embeddings (\ref{4.7})--(\ref{4.8}) are bounded by a
function of the condition number of $T$, i.e., the quantity which we
are trying to bound.

Ferreira in~\cite{FerreiraP-94} embeds a T\"{o}plitz matrix $T$ of order $n$,
defined by the symbol
$t = (t_0, \ldots, t_r, 0, \ldots, 0) \in \mathbb{R}^n$,
into the circulant $C$ of order $m = 2n$,
\begin{equation}
  C = \begin{bmatrix}
      T & S \\
      S & T
      \end{bmatrix},
  \label{4.9}
\end{equation}
where the symbol of the T\"{o}plitz matrix $S$ is
$s = (0, \ldots, 0, t_r, \ldots, t_1) \in \mathbb{R}^n$.

If we take $T = T_n^d$ from (\ref{3.1}), the only difference between
embeddings (\ref{4.1}) and (\ref{4.9}) is in exactly $n - d - r$ zero
diagonals, added as the first diagonals of $S$.
A sufficient condition for positive semidefiniteness of $C$ is given by
the next result.
\begin{thm}[Ferreira]
  Let $C$ be defined as in (\ref{4.9}), and let $b^T = [t_0, \ldots, t_{n - 1}]$,
  $c^T = [t_{n - 1}, \ldots, t_1]$. If $T$ is positive definite, and
  $|b^T T^{-1} c| < 1$, then $C$ is positive semidefinite.
\end{thm}
Once again, there is no obvious efficient way to verify whether the condition
$|b^T T^{-1} c| < 1$ is fullfiled or not.

%
%
\section{Conjecture about the minimal eigenvalues}
%
%
Extensive numerical testing has been conducted, by using
\textit{Mathematica\/} 7 from Wolfram Research, for the symbolic,
arbitrary-precision rational, and machine-precision floating-point
computations.  Whenever feasible, the full accuracy was maintained.  Owing
mostly to the elegance and the accuracy of these results, insight into
and the following conjecture about the spectral properties of the
collocation matrices and the corresponding periodizations were
obtained.

\begin{con}[The smallest eigenvalue of a circulant]
The circulant $C_m^d$ from (\ref{4.1}) is always positive definite, and
the index $\mu$ of its smallest eigenvalue $\lambda_\mu(C_m^d)$ is always
the integer nearest to $m / 2$, i.e.,
\begin{equation}
  \lambda_{\mu}(C_m^d) =
  \begin{cases}
    \displaystyle\lambda_{\frac{m \pm 1}{2}}(C_m^d)
    = t_0^d + 2 \sum_{j=1}^r(-1)^j t_j^d \cos\left(\frac{\pi j}{m}\right), &
    \hbox{$m$ odd,} \\
    \displaystyle\lambda_{\frac{m}{2}}(C_m^d) = t_0^d + 2 \sum_{j=1}^r(-1)^j t_j^d, &
    \hbox{$m$ even.}
  \end{cases}
  \label{5.1}
\end{equation}
\label{con5.1}
\end{con}
Figure~\ref{fig5.1} illustrates both cases of Conjecture~\ref{con5.1}.
\begin{figure}[hbt]
  \begin{center}
    \includegraphics[width=5.75cm]{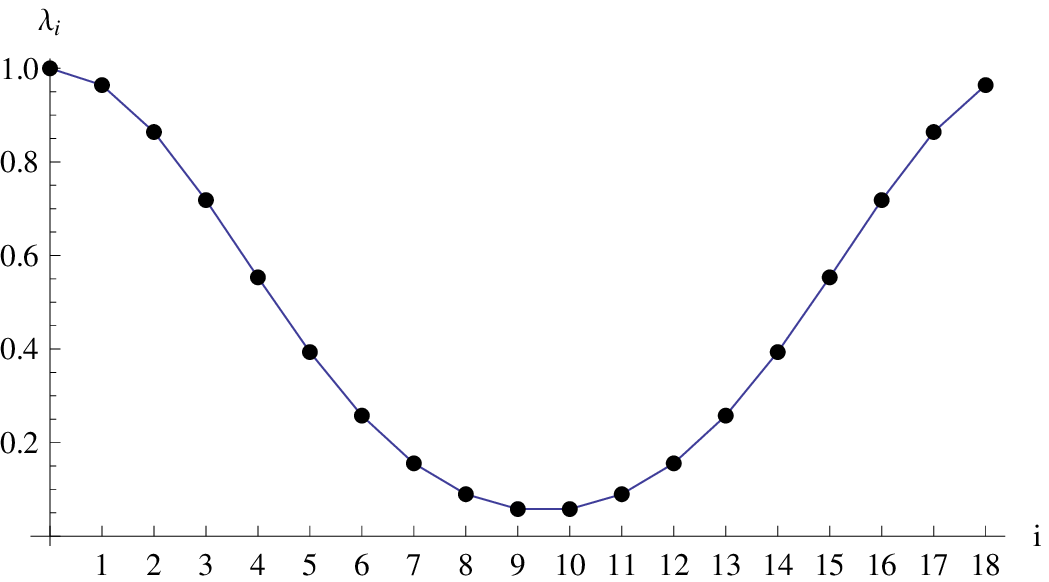} \qquad
    \includegraphics[width=5.75cm]{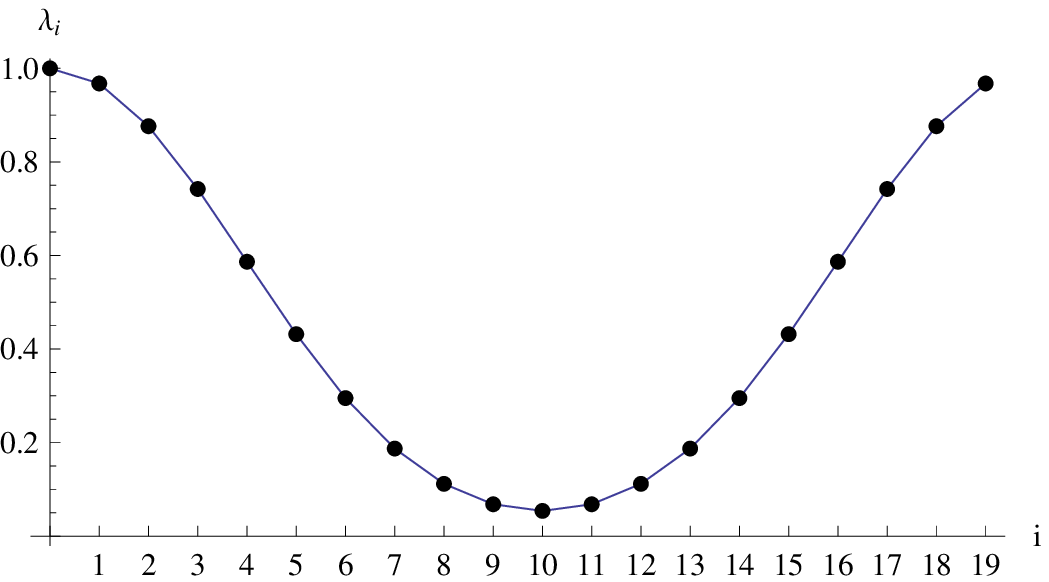}
  \end{center}
  \caption{The eigenvalues (black dots) $\lambda_k(C_m^d)$ for spline of
    degree $d = 7$ with $n = 23, 24$, respectively. The associated
    circulants have order $m = n - 7 + 3$, i.e., $19$ and $20$. Note that
    for $m = 20$ there is only one minimal eigenvalue, while for $m = 19$
    we have two minimal eigenvalues.}
  \label{fig5.1}
\end{figure}

For even $m$, $\lambda_{\mu}(C_m^d)$ (and, therefore, $\kappa_2(C_m^d)$)
depends solely on $d$, i.e., the order $m$ of a circulant is irrelevant
here.  Moreover, for $m$ odd and even alike, the limiting value of
$\lambda_{\mu}(C_m^d)$ is the same:
\begin{equation}
  \lambda_{\infty}^d := \lim_{m\to\infty} \lambda_{\mu}(C_m^d)
  = t_0^d + 2 \sum_{j=1}^r (-1)^j t_j^d.
  \label{5.2}
\end{equation}
Hence, the notation $\lambda_{\infty}^d$ is justified, since that
value is determined uniquely by the degree $d$ of the chosen cardinal
splines.  This is consistent with de~Boor's conjecture.

The equations (\ref{5.1}) and (\ref{5.2}) provide us with efficiently and
exactly computable estimates of the spectral condition numbers of large
collocation matrices $T_n^d$.  As demonstrated in Figure~\ref{fig5.2} and
Table~\ref{tbl5.1}, the smallest eigenvalues of the collocation matrices
converge rapidly and monotonically to the smallest eigenvalues of
the corresponding circulant periodizations $C_m^d$, as well as to the limiting
value~(\ref{5.2}).
\begin{figure}[hbt]
  \begin{center}
    \includegraphics[width=8cm]{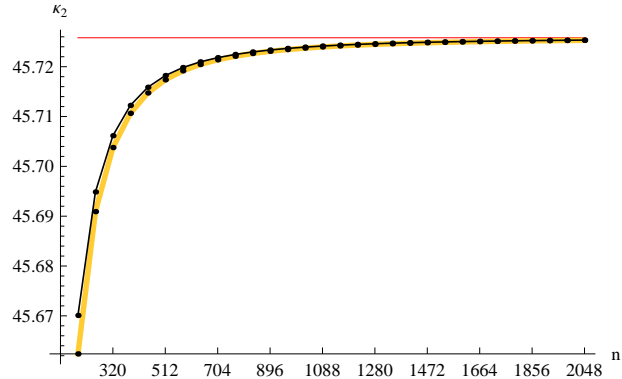}
  \end{center}
  \caption{Spectral condition numbers of T\"{o}plitz matrices $T_n^9$
    (lower, brighter line), and the circulant periodizations $C_m^9$
    (solid black line).
    The constant function denotes $1/\lambda_{\infty}^9$.}
  \label{fig5.2}
\end{figure}

It is worth noting that the spectral bounds obtained in such a way for
lower degrees ($d = 2, \ldots, 6$) of cardinal B-splines are quite
sharper than those established by the Ger\v{s}gorin circle theorem
(cf.~Table~\ref{tbl3.1} and Table~\ref{tbl5.1}), at no additional cost.
\begin{table}[hbt]
\begin{center}
\newcolumntype{I}{!{\vrule width 0.6pt}}
\newcommand\whline{\noalign{\global\savedwidth\arrayrulewidth
  \global\arrayrulewidth 0.6pt}%
  \hline
  \noalign{\global\arrayrulewidth\savedwidth}}

\renewcommand{\arraystretch}{1.25}
\begin{tabular}{cIccIccIcc}
\whline
$n\backslash d$ & $T_n^2$  & $C_m^2$  & $T_n^5$   & $C_m^5$   &  $T_n^6$ & $C_m^6$ \\
\whline
  64  & 1.998137 & 1.998758 & 7.466648 & 7.472749 & 11.72790 & 11.74214 \\
 128  & 1.999541 & 1.999694 & 7.492176 & 7.493492 & 11.78590 & 11.78866 \\
 256  & 1.999886 & 1.999924 & 7.498105 & 7.498410 & 11.79911 & 11.79971 \\
 512  & 1.999971 & 1.999981 & 7.499534 & 7.499607 & 11.80226 & 11.80240 \\
1024  & 1.999993 & 1.999995 & 7.499884 & 7.499902 & 11.80303 & 11.80306 \\
2048  & 1.999998 & 1.999999 & 7.499971 & 7.499975 & 11.80322 & 11.80322 \\
$1/\lambda_{\infty}^d$ & & 2.000000 & & 7.500000 & & 11.80328 \\
\whline
\end{tabular}
\begin{tabular}{cIccIccIcc}
\whline
$n\backslash d$ & $T_n^9$  & $C_m^9$  & $T_n^{21}$   & $C_m^{21}$   &  $T_n^{30}$ & $C_m^{30}$ \\
\whline
  64  & 45.04067 & 45.17179 & \hphantom{0}9012.21 & \hphantom{0}9543.49 & 371000.6 & 502472.1 \\
 128  & 45.57648 & 45.59721 & 10100.96 & 10150.47 & 569223.5 & 579852.3 \\
 256  & 45.69092 & 45.69486 & 10273.67 & 10279.58 & 594976.6 & 596037.0 \\
 512  & 45.71737 & 45.71822 & 10308.14 & 10309.00 & 599497.1 & 599628.0 \\
1024  & 45.72373 & 45.72393 & 10315.86 & 10316.01 & 600450.4 & 600469.5 \\
2048  & 45.72529 & 45.72534 & 10317.69 & 10317.72 & 600669.7 & 600673.0 \\
$1/\lambda_{\infty}^d$ & & 45.72581 & & 10318.28 & & 600739.5 \\
\whline
\end{tabular}
\end{center}
\caption{Comparison of the spectral condition numbers
  $\kappa_2(T_n^d)$ and $\kappa_2(C_m^d)$, for
  $d = 2, 5, 6, 9, 21, 30$, $n = 64, \ldots, 2048$, $m = n - d + r$,
  and $1/\lambda_{\infty}^d$.}
\label{tbl5.1}
\end{table}

Since $t_j^d$ are rational numbers,~(\ref{5.2}) is useful for the exact
computation of $\lambda_{\infty}^d$. But, in floating-point arithmetic,
the direct computation of $\lambda_{\infty}^d$ from~(\ref{5.2}) is
numerically unstable, as it certainly leads to severe cancellation.

It can be easily shown from~(\ref{2.1}) or~(\ref{2.4}) that the smallest
non-zero value of the cardinal B-spline of degree $d$ at an interpolation
node is:
\begin{displaymath}
  t_r^d = \begin{cases}
    N^d(1) = \frac{1}{d!},
      & \hbox{for odd\ $d$,}\\[6pt]
    N^d\left(\frac{1}{2}\right) = \frac{1}{2^d \cdot d!},
      & \hbox{for even $d$.}
  \end{cases}
\end{displaymath}
Moreover, all other values $t_j^d$ in~(\ref{5.2}) and, consequently,
$\lambda_{\infty}^d$ are integer multiples of~$t_r^d$.
With that in mind, yet another, somewhat surprising conjecture emerged from
the test results:
\begin{equation}
  \lambda_{\infty}^d = \begin{cases}
    t_r^d \cdot T_d = \frac{1}{d!} T_d, & \hbox{$d$ odd,} \\[6pt]
    t_r^d \cdot 2^d E_d = \frac{1}{d!} E_d, & \hbox{$d$ even,}
  \end{cases}
  \label{5.3}
\end{equation}
where, as in~\cite{Knuth-Buckholtz-67}, $T_n$ are the tangent numbers, and
$E_n$ are the Euler numbers, defined by the Taylor expansions of $\tan t$
and $\sec t$, respectively,
\begin{displaymath}
  \tan t = \sum_{n = 0}^\infty T_n \frac{t^n}{n!}, \quad
  \sec t = \sum_{n = 0}^\infty E_n \frac{t^n}{n!}.
\end{displaymath}
These numbers are also related to the sequences A000182 (the tangent or
``zag'' numbers), A000364 (the Euler or ``zig'' numbers) and A002436,
from~\cite{Sloane-2008}.

If true,~(\ref{5.3}) would be of significant practical merit, for there
exist very stable and elegant algorithms for calculation of $T_n$ and $E_n$
by Knuth and Buckholtz~\cite{Knuth-Buckholtz-67}.
So, it deserved an effort to find the proof.

A unifying framework for handling both cases is provided by the Euler
polynomials $E_n(x)$, defined by the following exponential generating
function (see~\cite[23.1.1, p.~804]{Abramowitz-Stegun-72})
\begin{equation}
  \frac{2 e^{xt}}{e^t + 1} = \sum_{n = 0}^\infty E_n(x) \frac{t^n}{n!},
  \label{5.4}
\end{equation}
which is valid for $|t| < \pi$.

First, note that $T_{2k} = E_{2k + 1} = 0$, for all $k \geq 0$.
The remaining nontrivial values can be expressed in terms of special
values of Euler polynomials.
For the tangent numbers, we have
\begin{equation}
  T_{2k + 1} = (-1)^k 2^{2k + 1} E_{2k + 1}(1), \quad k \geq 0.
  \label{5.5}
\end{equation}
This follows easily, by comparing the Taylor expansion of $1 + \tanh t$
\begin{displaymath}
  1 + \tanh t = \frac{2 e^{2t}}{e^{2t} + 1}
    = 1 + \sum_{k = 0}^\infty (-1)^k T_{2k + 1} \frac{t^{2k + 1}}{(2k + 1)!}
\end{displaymath}
and~(\ref{5.4}), with $x = 1$ and $2t$, instead of $t$.
Similarly, by comparing the Taylor expansion of $\sech t$
\begin{displaymath}
  \sech t = \frac{2 e^t}{e^{2t} + 1}
    = \sum_{k = 0}^\infty (-1)^k E_{2k} \frac{t^{2k}}{(2k)!}
\end{displaymath}
and~(\ref{5.4}), with $x = 1/2$ and $2t$, instead of $t$, we get
\begin{equation}
  E_{2k} = (-1)^k 2^{2k} E_{2k} \left( \frac{1}{2} \right), \quad k \geq 0.
  \label{5.6}
\end{equation}

The following identities will also be needed in the proof of~(\ref{5.3}).
\begin{lem}
Let $d \geq 0$ be a non-negative integer. Then
\begin{align}
  \sum_{\ell = 0}^{d + 1} (-1)^\ell \binom{d + 1}{\ell} E_n(\ell) = 0,
  \label{5.7} \\
  \sum_{\ell = 0}^{d + 1} (-1)^\ell \binom{d + 1}{\ell} E_n(\ell + 1) = 0,
  \label{5.8}
\end{align}
for all $n = 0, \ldots, d$.
\end{lem}

\begin{proof}
Consider the function $g_d$ defined by
\begin{displaymath}
  g_d(t) := \frac{2 (1 - e^t)^{d + 1}}{e^t + 1}
    = \sum_{\ell = 0}^{d + 1} (-1)^\ell \binom{d+1}{\ell}
      \frac{2 e^{\ell t}}{e^t + 1}.
\end{displaymath}
From~(\ref{5.4}) with $x = \ell$, the Taylor expansion of $g_d$ can be
written as
\begin{displaymath}
  g_d(t)
    = \sum_{n = 0}^\infty \left[
      \sum_{\ell = 0}^{d + 1} (-1)^\ell \binom{d + 1}{\ell} E_n(\ell)
      \right] \frac{t^n}{n!},
\end{displaymath}
so
\begin{displaymath}
  D^n g_d(t) \big|_{t = 0}
    = \sum_{\ell = 0}^{d + 1} (-1)^\ell \binom{d + 1}{\ell} E_n(\ell),
    \quad n \geq 0.
\end{displaymath}
On the other hand, the Leibniz rule gives
\begin{displaymath}
  D^n g_d(t) = \sum_{m = 0}^{n} \binom{n}{m}
    D^m \left[ (1 - e^t)^{d + 1} \right] \,
    D^{n - m} \left[ \frac{2}{e^t + 1} \right].
\end{displaymath}
If $n \leq d$, then $D^m \left[ (1 - e^t)^{d + 1} \right]$ is always
divisible by $(1 - e^t)$. Hence,
\begin{displaymath}
  D^n g_d(t) \big|_{t = 0} = 0, \quad n = 0, \ldots, d,
\end{displaymath}
which proves the first identity~(\ref{5.7}).

The second one follows similarly, by considering
\begin{displaymath}
  h_d(t) := g_d(t) - g_{d + 1}(t) = \frac{2 e^t (1 - e^t)^{d + 1}}{e^t + 1}
    = \sum_{\ell = 0}^{d + 1} (-1)^\ell \binom{d+1}{\ell}
    \frac{2 e^{(\ell + 1) t}}{e^t + 1}.
\end{displaymath}
The Taylor expansion of $h_d$ is then given by
\begin{displaymath}
  h_d(t)
    = \sum_{n = 0}^\infty \left[
    \sum_{\ell = 0}^{d + 1} (-1)^\ell \binom{d + 1}{\ell} E_n(\ell + 1)
    \right] \frac{t^n}{n!}.
\end{displaymath}
If $n \leq d$, from the first part of the proof, it follows immediately that
\begin{displaymath}
  D^n h_d(t) \big|_{t = 0}
    = D^n g_d(t) \big|_{t = 0} - D^n g_{d + 1}(t) \big|_{t = 0} = 0,
\end{displaymath}
which proves~(\ref{5.8}).
\end{proof}

Finally, we are ready to prove the conjecture~(\ref{5.3}).
\begin{thm}[Relation to integer sequences]
The following holds for all cardinal B-spline degrees $d \geq 0$
\begin{displaymath}
  \lambda_{\infty}^d = \frac{1}{d!} \cdot \begin{cases}
    T_d, & \hbox{$d$ odd,} \\
    E_d, & \hbox{$d$ even.}
  \end{cases}
\end{displaymath}
\end{thm}

\begin{proof}
To simplify the notation, let $L_d := d! \, \lambda_{\infty}^d$. Due to the
symmetry of interpolation nodes, the sum in~(\ref{5.2}) can be written as
\begin{displaymath}
  \lambda_{\infty}^d = \sum_{j = -r}^r (-1)^j t_j^d, \quad
    t_j^d = N^d \left( j + \frac{d + 1}{2} \right), \quad j = -r, \ldots, r,
\end{displaymath}
where $r = \lfloor d/2 \rfloor$. From~(\ref{2.1}) and~(\ref{2.2}), it
follows that
\begin{displaymath}
  t_j^d = \frac{1}{d!} \sum_{\ell = 0}^{d + 1} (-1)^\ell \binom{d + 1}{\ell}
    \left( j + \frac{d + 1}{2} - \ell \right)_{+}^d.
\end{displaymath}
Then
\begin{equation}
  L_d
    = \sum_{j = -r}^r (-1)^j
    \sum_{\ell = 0}^{d + 1} (-1)^\ell \binom{d + 1}{\ell}
    \left( j - \ell + \frac{d + 1}{2} \right)_{+}^d.
  \label{5.9}
\end{equation}

Let $d$ be odd, $d = 2k + 1$, with $k \geq 0$. Then $r = k$ and
$(d + 1)/2 = k + 1$, so~(\ref{5.9}) becomes
\begin{displaymath}
  L_{2k + 1}
    = \sum_{j = -k}^k (-1)^j
    \sum_{\ell = 0}^{2k + 2} (-1)^\ell \binom{2k + 2}{\ell}
    \left( j - \ell + k + 1 \right)_{+}^{2k + 1}.
\end{displaymath}
From the definition of truncated powers with positive exponents, the second
sum contains only the terms with $j - \ell + k + 1 > 0$, i.e., for
$l \leq j + k$. By changing the order of summation, we get
\begin{displaymath}
  L_{2k + 1}
    = \sum_{\ell = 0}^{2k} (-1)^\ell \binom{2k + 2}{\ell}
    \sum_{j = \ell - k}^k (-1)^j
    \left( j - \ell + k + 1 \right)^{2k + 1}.
\end{displaymath}
Then we shift $j$ by $k - \ell + 1$, so that $j$ starts at $1$, to obtain
\begin{displaymath}
  L_{2k + 1}
    = (-1)^k \sum_{\ell = 0}^{2k} (-1)^\ell \binom{2k + 2}{\ell}
    \sum_{j = 1}^{2k + 1 - \ell} (-1)^{(2k + 1 - \ell) - j}
    j^{2k + 1}.
\end{displaymath}
The second sum can be simplified as
(see~\cite[23.1.4, p.~804]{Abramowitz-Stegun-72})
\begin{displaymath}
  \sum_{j = 1}^{2k + 1 - \ell} (-1)^{(2k + 1 - \ell) - j} j^{2k + 1}
    = \frac{1}{2} \left(
    E_{2k + 1}(2k + 2 - \ell) + (-1)^{2k + 2 - \ell} E_{2k + 1}(1) \right).
\end{displaymath}
Hence
\begin{displaymath}
  L_{2k + 1}
    = \frac{(-1)^k}{2} \left[
    \sum_{\ell = 0}^{2k}
      (-1)^\ell \binom{2k + 2}{\ell} E_{2k + 1}(2k + 2 - \ell)
    + E_{2k + 1}(1) \sum_{\ell = 0}^{2k} \binom{2k + 2}{\ell}
  \right].
\end{displaymath}
By reversing the summation, from~(\ref{5.7}) with $d = 2k + 1$ and $n = d$,
we conclude that
\begin{align*}
  \sum_{\ell = 0}^{2k}
    (-1)^\ell \binom{2k + 2}{\ell} E_{2k + 1}(2k + 2 - \ell) & =
  \sum_{\ell = 2}^{2k + 2}
    (-1)^\ell \binom{2k + 2}{\ell} E_{2k + 1}(\ell) \\
  & =
  - \sum_{\ell = 0}^1 (-1)^\ell \binom{2k + 2}{\ell} E_{2k + 1}(\ell).
\end{align*}
Since $E_{2k + 1}(0) = - E_{2k + 1}(1)$, by using~(\ref{5.5}), we have
\begin{displaymath}
  L_{2k + 1}
    = \frac{(-1)^k}{2}
    E_{2k + 1}(1) \sum_{\ell = 0}^{2k + 2} \binom{2k + 2}{\ell}
    = (-1)^k 2^{2k + 1} E_{2k + 1}(1)
    = T_{2k + 1}.
\end{displaymath}
This proves the claim for odd values of $d$.

Let $d$ be even, $d = 2k$, with $k \geq 0$. For $d = 0$,
it is obvious that $L_0 = t_0^0 = 1 = E_0$, so we may assume
that $k > 0$.
Then $r = k$ and
$(d + 1)/2 = k + 1/2$, so~(\ref{5.9}) becomes
\begin{displaymath}
  L_{2k}
    = \sum_{j = -k}^k (-1)^j
    \sum_{\ell = 0}^{2k + 1} (-1)^\ell \binom{2k + 1}{\ell}
    \left( j - \ell + k + \frac{1}{2} \right)_{+}^{2k}.
\end{displaymath}
The second sum contains only the terms with $j - \ell + k + 1/2 > 0$,
i.e., for $l \leq j + k$.
By exactly the same transformation as before, we arrive at
\begin{displaymath}
  L_{2k}
    = (-1)^k \sum_{\ell = 0}^{2k} (-1)^\ell \binom{2k + 1}{\ell}
    \sum_{j = 1}^{2k + 1 - \ell} (-1)^{(2k + 1 - \ell) - j}
    \left( j - \frac{1}{2} \right)^{2k}.
\end{displaymath}
Now we expand the last factor in terms of powers of $j$. Then $L_{2k}$
can be written as
\begin{equation}
  L_{2k} = (-1)^k \sum_{n = 0}^{2k} \binom{2k}{n}
    \left( - \frac{1}{2} \right)^{2k - n} S_{2k, n},
  \label{5.10}
\end{equation}
with
\begin{displaymath}
  S_{2k, n} = \sum_{\ell = 0}^{2k} (-1)^\ell \binom{2k + 1}{\ell}
    \sum_{j = 1}^{2k + 1 - \ell} (-1)^{(2k + 1 - \ell) - j} j^n,
    \quad n = 0, \ldots, 2k.
\end{displaymath}
Like before, the second sum can be simplified as
\begin{displaymath}
  \sum_{j = 1}^{2k + 1 - \ell} (-1)^{(2k + 1 - \ell) - j} j^n
    = \frac{1}{2} \left(
    E_n(2k + 2 - \ell) + (-1)^{2k + 2 - \ell} E_n(1) \right),
\end{displaymath}
which gives
\begin{displaymath}
  S_{2k, n}
    = \frac{1}{2} \left[
    \sum_{\ell = 0}^{2k}
      (-1)^\ell \binom{2k + 1}{\ell} E_n(2k + 2 - \ell)
    + E_n(1) \sum_{\ell = 0}^{2k} \binom{2k + 1}{\ell}
  \right].
\end{displaymath}
By reversing the summation, from~(\ref{5.8}) with $d = 2k$,
for $n = 0, \ldots, d$, we see that
\begin{displaymath}
  \sum_{\ell = 0}^{2k}
    (-1)^\ell \binom{2k + 1}{\ell} E_n(2k + 2 - \ell)
   = - \sum_{\ell = 1}^{2k + 1}
    (-1)^\ell \binom{2k + 1}{\ell} E_n(\ell + 1)
  = E_n(1).
\end{displaymath}
Therefore,
\begin{displaymath}
  S_{2k, n}
    = \frac{1}{2} E_n(1) \sum_{\ell = 0}^{2k + 1} \binom{2k + 1}{\ell}
    = 2^{2k} E_n(1).
\end{displaymath}
From~(\ref{5.10}) we obtain
\begin{displaymath}
  L_{2k} = (-1)^k 2^{2k} \sum_{n = 0}^{2k} \binom{2k}{n}
    E_n(1) \left( -\frac{1}{2} \right)^{2k - n}.
\end{displaymath}
Finally, by using~\cite[23.1.7, p.~804]{Abramowitz-Stegun-72})
\begin{displaymath}
  \sum_{n = 0}^{2k} \binom{2k}{n} E_n(1) \left( -\frac{1}{2} \right)^{2k - n}
  = E_{2k} \left( \frac{1}{2} \right).
\end{displaymath}
Together with~(\ref{5.6}), this gives
\begin{displaymath}
  L_{2k} = (-1)^k 2^{2k} E_{2k} \left( \frac{1}{2} \right) = E_{2k}.
\end{displaymath}
This completes the proof for even values of $d$.
\end{proof}

We would like to conclude with an observation that, to the best of our
knowledge, scarcely any result could be found about sufficient
conditions for the non-negativeness of the DFT in terms of its
coefficients, apart from the classical result of Young and Kolmogorov
(cited in Zygmund~\cite[page~109]{Zygmund-35}):
\begin{thm}
For a convex sequence $(a_n,n\in\mathbb{N})$, where
$\displaystyle\lim_{n\to\infty}a_n=0$, the sum
\begin{displaymath}
  \frac{1}{2} a_0 + \sum_{n=1}^{\infty} a_n \cos\left( nx \right)
\end{displaymath}
converges (save for $x = 0$), and is non-negative.
\end{thm}
Here, a sequence is convex if $\Delta^2 a_n \ge 0$ for all $n$, with
$\Delta a_n = a_n - a_{n+1}$.

Convexity is not fulfilled in the case of cardinal B-spline
coefficients, since there is always one inflection point on each slope
of the spline.  And yet, our numerical experiments strongly suggest
that the class of series with a positive DFT is worth investigating
further, for the theoretical and practical reasons alike.

\end{document}